\documentclass[10pt]{amsart}

\usepackage[all]{xy}
\usepackage{hyperref}
\hypersetup{
    colorlinks=true,      
    linkcolor=blue,          
    citecolor=blue,       
    filecolor=blue,      
    urlcolor=blue         
}
\usepackage[lite]{amsrefs}
\usepackage{ amssymb }
\usepackage{extarrows}
\usepackage{todonotes}
\usepackage{ stmaryrd }
\usepackage{graphicx, amssymb, color,pdfsync}
\usepackage[all]{xy}
\usepackage{hyperref}
\usepackage{enumerate}
\usepackage{extarrows}
\usepackage{pgf,tikz}
\numberwithin{equation}{section} 

\usepackage{todonotes}
\usepackage{nicefrac}

\newcommand{\rank}{\operatorname{rank}}

\newcommand{\supp}{\operatorname{supp}}

\newcommand{\Spec}{\operatorname{Spec}}

\newcommand{\ord}{\operatorname{ord}}

\newcommand{\loc}{\operatorname{loc}}

\newcommand{\relint}{\operatorname{relint}}

\newcommand{\divv}{\operatorname{div}}

\newcommand{\Pol}{\operatorname{Pol}}
\newcommand{\CaDiv}{\operatorname{CaDiv}}

\newcommand{\Proj}{\operatorname{Proj}}

\newcommand{\coef}{\operatorname{coef}}
\newcommand{\rec}{\operatorname{rec}}

\newcommand{\pp}{\mathbb{P}}

\newcommand{\qq}{\mathbb{Q}}
\newcommand{\zz}{\mathbb{Z}}

\newcommand{\kk}{\mathbb{K}}

\newcommand{\TT}{\mathbb{T}}
\newcommand{\ZZ}{\mathbb{Z}}
\newcommand{\QQ}{\mathbb{Q}}
\newcommand{\DD}{\mathcal{D}}
\newcommand{\OO}{\mathcal{O}}

\newtheorem{introthm}{Theorem}

\newtheorem{theorem}{Theorem}[section]
\newtheorem{lemma}[theorem]{Lemma}

\newtheorem*{theorem*}{Theorem}

\theoremstyle{definition}

\newtheorem{notation}[theorem]{Notation}
\newtheorem{definition}[theorem]{Definition}
\newtheorem{example}[theorem]{Example}

\newtheorem{construction}[theorem]{Construction}

\newtheorem{remark}[theorem]{Remark}

\theoremstyle{remark}

\numberwithin{equation}{section}

\usepackage[norelsize]{algorithm2e}

\begin{document}

\title[Cohen-Macaulay Du Bois singularities with a torus action]{Cohen-Macaulay Du Bois singularities \protect\\
 with a torus action of complexity one}

\author{Antonio Laface}

\address{Departamento de Matem\'atica, Universidad de Concepci\'on,
  Casilla 160-C, Concepci\'on, Chile} \email{alaface@udec.cl}

\author{Alvaro Liendo}

\address{Instituto de Matem\'atica y F\'isica, Universidad de Talca,
  Casilla 721, Talca, Chile} \email{aliendo@inst-mat.utalca.cl}

\author{Joaqu\'\i n Moraga}

\address{Department of Mathematics, University of Utah, 155 S 1400 E,
  Salt Lake City, UT 84112} \email{moraga@math.utah.edu}


\thanks{ The first author was partially supported by Proyecto FONDECYT
  Regular N. 1150732.  The second author was partially supported by
  Proyecto FONDECYT Regular N. 1160864.}

\maketitle

\medskip 

\begin{abstract}
  Using Altmann-Hausen-S{\"u}\ss\ description of
  $\mathbb{T}$-varieties via divisorial fans and
  K\'ovacs-Schwede-Smith characterization of Du Bois singularities, we
  study Cohen-Macaulay Du Bois $\mathbb{T}$-singularities of complexity one. 
  We exhibit cohomological criteria for a $\mathbb{T}$-variety to be Cohen-Macaulay and Du Bois
  in terms of polyhedral divisors.
  We give an example of a Cohen-Macaulay Du Bois singularity  of complexity one which does not have rational singularities.
\end{abstract}

\setcounter{tocdepth}{1}
\tableofcontents

\section*{Introduction}

We study the singularities of $\mathbb{T}$-varieties, i.e. normal
algebraic varieties endowed with an effective action of an algebraic
torus $\mathbb{T}=(\kk^*)^k$, where $\kk$ is an algebraically closed
field of characteristic zero.  Given a $\mathbb{T}$-variety $X$ we
define its {\em complexity} to be $\dim(X)-\dim(\mathbb{T})$.
$\mathbb{T}$-varieties of complexity zero are called {\em toric}
varieties and they admit a well-known combinatorial description in
terms of fans of polyhedral cones~\cites{KKMS73,CLS}. Such combinatorial
description can be generalized to higher complexity. The complexity
one case was considered in ~\cites{KKMS73,FZ,Tim08}. Altmann,
Hausen and S{\"u}\ss\ generalize this combinatorial description for
arbitrary complexity using the language of polyhedral divisors and
divisorial fans on normal projective algebraic varieties~\cites{AH05,
  AHH}. We will use the language of polyhedral divisors in this paper. In this context, an
affine $\TT$-variety $X$ is determined by a pair $(Y,\mathcal{D})$,
where $Y$ is the Chow quotient for the $\TT$-action on $X$ and
$\mathcal{D}$ is a p-divisor on $Y$, i.e., a formal finite sum of
prime divisors on $Y$ whose coefficients are convex polyhedra. Recall
that the Chow quotient of $X$ by the action of $\mathbb{T}$ is the
closure of the set of general $\mathbb{T}$-orbit closures seen as
points in the Chow variety (see Section~\ref{sec1} for the details).

Using toric resolution of singularities is it possible to prove that
every normal toric variety $X$ has log terminal singularities
~\cite{CLS}*{Proposition 11.4.24}.  Therefore, normal toric varieties
have rational singularities and then are Cohen-Macaulay and Du Bois, 
see~\cite{KM98}. However, in higher complexity none of this
statements is true if we do not impose conditions on the defining
combinatorial data. For instance, a simple computation shows that the
affine cone $X_C$ over a plane curve $C$ of genus $g$ is a normal
$\TT$-surface of complexity one which is log canonical if and only if
$g$ equals $0$ or $1$. Furthermore, in the case that $C$ is a curve of genus $g\geq 1$, the
singularity at the vertex is not rational~\cite{LS} and so by virtue
of~\cite{KK10}, $X$ is an example of a $\TT$-variety of complexity one
which is Cohen-Macaulay and Du Bois but has a singularity that is not
rational.

Therefore it is worthwhile to study characterizations of this kind of
singularities for $\TT$-varieties of higher complexity in terms of the
defining p-divisor. There are such characterizations of rational
singularities and $\qq$-Gorenstein singularities and partial
characterizations of Cohen-Macaulayness in terms of the defining
p-divisors (see~\cite{LS}). In the case of positive characteristic,
there are also characterizations of $F$-split and $F$-regular
$\mathbb{T}$-varieties of complexity one in terms of p-divisors (see~\cite{AIS}).

Cohen-Macaulay singularities are the most natural singularities which
have the same cohomological behavior as smooth varieties, for
instance vanishing theorems still hold for varieties with
Cohen-Macaulay singularities (see~\cite{KM98}).  Du Bois singularities
also play an important role in algebraic geometry since these are the
singularities appearing in the minimal model program and moduli
theory (see~\cites{Sch07,SS11}).  In this article, we study
$\mathbb{T}$-varieties with Cohen-Macaulay and Du Bois singularities.

Throughout the introduction, we restrict ourselves to the case of rational $\mathbb{T}$-varieties of complexity one.
We briefly recall the data defining a rational complexity one affine $\mathbb{T}$-variety (we refer the reader to Section~\ref{sec1}
and for the details).
These varieties are determined by a finite set of points
$p_1,\dots,p_r\in \pp^1$
and a finite set of polyhedra
$\Delta_{p_1},\dots,\Delta_{p_r}\subset N_\qq := N \otimes \qq \simeq \qq^r$ with a common recession cone $\sigma\subset N_\qq$. 
Let $M:={\rm Hom}(N,\zz)$.
This data induces a function 
$\mathcal{D} \colon M \rightarrow {\rm CaDiv}_\qq(\pp^1)$
defined by 
\[
\mathcal{D}(u) := \sum_{i=1}^r \min_{v\in \mathcal{V}_i} \langle u, v\rangle p_i
\]
where $\mathcal{V}_i$ is the set of vertices of $\Delta_{p_i}$.
We say that $\mathcal{D}$ is a {\em p-divisor over $\pp^1$} if 
$\sum_{i=1}^r \Delta_{p_i}\subset \sigma$;
this is a particular case of Definition~\ref{def:p-div}.
There exists a rational $\mathbb{T}$-variety of complexity one $X(\mathcal{D})$ associated to the p-divisor $\mathcal{D}$ (see Construction~\ref{con:t-var} and Definition~\ref{def:t-var}).
Furthermore, every rational affine $\mathbb{T}$-variety of complexity one is isomorphic to such a $X(\mathcal{D})$~\cite{AH05}.
We say that a ray $\rho \in \sigma(1)$ is {\em big} if 
$\sum_{i=1}^r \Delta_{p_i} \cap \rho =\emptyset$.
Otherwise, we say that $\rho\in \sigma(1)$ is {\em not big}.
Our first theorem, is an application of Kempf criteria for rational singularities~\cite{Kempf} in the context of rational $\mathbb{T}$-varieties of complexity one.

\begin{introthm}\label{thm:comp-1-rat}
 Let $\mathcal{D}$ be a p-divisor on $(\pp^1, N)$. 
 Assume that $X(\mathcal{D})$ is Cohen-Macaulay.
 Then, $X(\mathcal{D})$ is rational if and only if
 $\deg \lfloor \mathcal{D}(u)\rfloor \geq -1$ for every $u$ in the set \[
 \{ u\in M \mid 
 \langle u, \rho \rangle \leq -1 \text{ for all $\rho$ big and } 
 \langle u,\rho \rangle \geq 0 
 \text{ for al least one $\rho$ not big}
 \}. 
 \]
\end{introthm}

In a similar vein, 
using Kov\'acs-Schwede-Smith characterization
of Du Bois singularities \cite{KSS}*{Theorem 1}, we prove the following
result:

\begin{introthm}\label{thm:comp-1-DB}
Let $\mathcal{D}$ be a p-divisor on 
$(\pp^1,N)$. Assume that $X(\mathcal{D})$ is Cohen-Macaulay.
Then, $X(\mathcal{D})$ has Du Bois singularities
if and only if $\deg\lfloor \mathcal{D}(u)\rfloor \geq -1$ for every $u$ in the set
\[
\{
u\in M\mid 
\langle u,\rho \rangle \leq -1 \text{  for all $\rho$ big and }
\langle u, \rho \rangle \geq 1 \text{ for at least one $\rho$ not big }
\}. 
\]
\end{introthm}

The paper is organized as follows: in Section~\ref{sec1}, we introduce
the basic notation of $\mathbb{T}$-varieties via p-divisors.
In 
Section~\ref{gensing}, we introduce the classes of singularities needed in
the article, i.e., rational, Cohen-Macaulay, Du Bois, and log canonical.  
In Section~\ref{sec2}, we prove some preliminary
results regarding the canonical divisor of $\mathbb{T}$-varieties
and sections of invariant line bundles.
In Section~\ref{sec25} we prove
Theorem~\ref{thm:rat} 
and Theorem~\ref{thm:DB} 
which are generalizations of 
Theorem~\ref{thm:comp-1-rat}
and Theorem~\ref{thm:comp-1-DB}
to higher complexity.

\subsection*{Acknowledgements}
We would like to thank Karl Schwede and Linquan Ma for useful comments
about Du Bois singularities.
We wish to thank the
anonymous referees of an early version of this work for useful comments and suggestions that helped the authors improve and correct the results.

\section{$\mathbb{T}$-Varieties via polyhedral divisors}\label{sec1}

We work over an algebraically closed field $\kk$ of characteristic
zero.  In this section, we briefly introduce the language of
p-divisors introduced in ~\cite{AH05} and ~\cite{AHH}.  We begin by
recalling the definition of polyhedral divisors, p-divisors and their
connection with affine $\mathbb{T}$-varieties.

Let $N$ be a finitely generated free abelian group of rank $k$ and
denote by $M$ its dual. We also let $N_\qq$ and $M_\qq$ be the
corresponding $\qq$-vector spaces obtained by tensoring $N$ and $M$
with $\qq$ over $\mathbb{Z}$, respectively.  We denote by
$\mathbb{T} : = \Spec(\kk [M] )$ the $k$-dimensional algebraic torus.
Given a polyhedron $\Delta \subset N_\qq$ we will denote its {\em
  recession cone} by
\[
\rec(\Delta) :=\{ v\in N_\qq \mid v + \Delta \subset \Delta\},
\]
where $+$ denotes the Minkowski sum, 
which is defined by 
\[ 
\Delta_1 + \Delta_2 := \{ w \in N_\qq \mid 
w=v_1+v_2, v_1\in \Delta_1 \text{ and } v_2 \in \Delta_2\}.
\] 
Given a pointed polyhedral cone $\sigma\subset N_\qq$ 
we can define a semigroup with underlying set 
\[
\Pol_\qq(N,\sigma) :=\{ \Delta \subset N_\qq \mid \text{ $\Delta$ is a
polyhedron with $\rec(\Delta)=\sigma$}\}
\]
and addition being the Minkowski sum. The neutral element in the
semigroup is the cone $\sigma$. The elements in $\Pol_\qq(N,\sigma)$
are called $\sigma$-polyhedra.  In what follows we consider the
semigroup $\Pol^+_\qq(N,\sigma)$ which is the extension of the above
semigroup obtained by including the element $\emptyset$ which is an
absorbing element with respect to the addition, meaning that:
\[
  \emptyset + \Delta :=\emptyset \text{ for all
    $\Delta \in \Pol^+_\qq(N,\sigma)$}.
\]
Given a normal projective variety $Y$ we denote by
$\CaDiv_{\geq 0}(Y)$ the semigroup of effective Cartier divisors on
$Y$. We define a {\em polyhedral divisor on $(Y,N)$ with recession
  cone $\sigma$} to be an element of the semigroup
\[
\Pol_\qq^+(N,\sigma) \otimes_{\zz_{\geq 0}} \CaDiv_{\geq 0}(Y).
\]
Observe that any polyhedral divisor can be written as 
\[
\mathcal{D}= \sum_{Z} \Delta_Z \cdot Z,
\]
where the sum runs over prime divisors of $Y$, the coefficients
$\Delta_Z$ are either $\sigma$-polyhedra or the empty set and all but
finitely many $\Delta_Z$ are the neutral element $\sigma$. The
recession cone of a polyhedral divisor $\mathcal{D}$ is defined as the
recession cone of any non-empty coefficient and is denoted by
$\sigma(\mathcal{D})$.  The {\em locus of $\mathcal{D}$} is the open
set
\[
\loc(\mathcal{D}) := Y \setminus \bigcup_{\Delta_Z=\emptyset} Z
\] 
and we say that $\mathcal{D}$ has {\em complete locus} if the equality
$\loc(\mathcal{D})=Y$ holds, meaning that all the coefficients
$\Delta_Z$ are nonempty $\sigma(\mathcal{D})$-polyhedra.  A polyhedral
divisor $\mathcal{D}$ on $(Y,N)$ with recession cone $\sigma$, defines
a homomorphism of semigroups as follows
\[
\mathcal{D} \colon \sigma^\vee \rightarrow 
\CaDiv_\qq ( \loc(\mathcal{D})) 
\]
\[
u\mapsto \sum_{\Delta_Z \neq \emptyset} 
\min \langle \Delta_Z, u \rangle Z|_{\loc(\mathcal{D})},
\]
where by abuse of notation we denote by $\mathcal{D}$ both the
polyhedral divisor and the homomorphism of semigroups.  Observe that
this homomorphism is well-defined since all the polyhedra $\Delta_Z$
have recession cone $\sigma(\mathcal{D})$ and then the minimum
appearing in the definition always exists.  Moreover, for any
$u\in \sigma^\vee$ we have that $\mathcal{D}(u)$ is a $\qq$-divisor in
$\loc(\mathcal{D})$ whose support contained in
\[
\supp(\mathcal{D}) := 
\loc(\mathcal{D}) \cap \bigcup_{\Delta_Z \neq \emptyset} Z. 
\]

Recall that a variety $Y$ is called semiprojective if the natural map
$Y\rightarrow H^0(Y,\OO_Y)$ is projective. In particular, affine and
projective varieties are semiprojective. Furthermore, blow-ups of
semiprojective varieties are also semiprojective.

\begin{definition}\label{def:p-div}
  Let $\mathcal{D}$ be a polyhedral divisor on $(Y,N)$ 
  with recession
  cone $\sigma$.  We say that $\mathcal{D}$ is a {\em proper
    polyhedral divisor}, 
    or {\em p-divisor} for short, if
  $\loc(\mathcal{D})$ is semiprojective, $\mathcal{D}(u)$ is semiample
  for $u \in \sigma^\vee$ and $\mathcal{D}(u)$ is big for
  $u\in \relint(\sigma^\vee)$.  Note that this condition holds
  whenever the locus of the polyhedral divisor is affine.
\end{definition}

\begin{remark}
In what follows, given a $\qq$-divisor $D$ on an algebraic variety $X$
we will write $\mathcal{O}_X(D)$ for $\mathcal{O}_X(\lfloor D \rfloor)$.
\end{remark}

\begin{construction}\label{con:t-var}
Given a p-divisor $\mathcal{D}$ on $(Y,N)$ 
with recession cone $\sigma$ 
we have an induced sheaf of $\mathcal{O}_{\loc(\mathcal{D})}$-algebras
\[
\mathcal{A}(\mathcal{D}) := \bigoplus_{u \in \sigma^\vee \cap M} 
\mathcal{O}_{\loc(\mathcal{D})}(\mathcal{D}(u))\chi^u.
\]
Regard that $\mathcal{A}(\mathcal{D})$ is indeed a 
sheaf of $\mathcal{O}_{\loc(\mathcal{D})}$-algebras 
since the inequality
\[
\mathcal{D}(u)+\mathcal{D}(u')\leq \mathcal{D}(u+u')
\]
holds for every $u,u'\in \sigma^\vee$ by convexity of the polyhedral
coefficient.  We denote by $\widetilde{X}(\mathcal{D})$ the relative
spectrum of $\mathcal{A}(\mathcal{D})$ and by $X(\mathcal{D})$ the
spectrum of the ring of global sections, both $X(\mathcal{D})$ and
$\widetilde{X}(\mathcal{D})$ comes with a $\mathbb{T}$-action induced
by the $M$-grading.  The main theorem of ~\cite{AH05} states that
every $n$-dimensional normal affine variety $X$ with an effective
action of a $k$-dimensional torus $\mathbb{T}$ is
$\mathbb{T}$-equivariantly isomorphic to $X(\mathcal{D})$ for some
$p$-divisor $\mathcal{D}$ on $(Y,N)$, where $Y$ is a projective normal
variety of dimension $n-k$ and $\rank(N)=k$.  In this situation we
have two natural morphisms
\[
  \xymatrix@C=20pt{ & \loc(\mathcal{D}) &
    \widetilde{X}(\mathcal{D})\ar[r]^-{r}
    \ar[l]_-{\pi} & X(\mathcal{D}), \\
  }
\]
where $\pi$ is the good quotient induced by the inclusion of sheaves
$\mathcal{O}_{\loc(\mathcal{D})} \hookrightarrow
\mathcal{A}(\mathcal{D})$ and $r$ is the $\mathbb{T}$-equivariant
birational map induced by taking global sections of
$\mathcal{A}(\mathcal{D})$. Moreover, every affine $\TT$-variety can be
recovered from a p-divisor on the Chow quotient by the $\TT$-action (see \cite[Definition~8.7]{AH05}).
\end{construction} 

\begin{definition}\label{def:t-var} 
Let $\mathcal{D}$ be a p-divisor on $(Y,N)$.
  The affine $\mathbb{T}$-variety $X(\mathcal{D})$ is called 
  the $\mathbb{T}$-variety {\em associated to the p-divisor $\mathcal{D}$}.
\end{definition}

Let $\mathcal{D}$ be a p-divisor. A ray $\rho \in \sigma(\mathcal{D})$ is called  big if $\mathcal{D}(u)$ is big for $u$ in the relative interior of the cone
$\sigma(\mathcal{D})^\vee \cap \rho^\perp$.  A vertex $v$ in $\Delta_Z$ is called big if $\mathcal{D}(u)|_Z$ is big for every
$u$ in the relative interior of the cone
$\left\{u \mid\langle u, w -v \rangle > 0 \mbox{ for all } w \in \Delta_Z \right\}$.

\section{Singularities of normal varieties}\label{gensing}

In this section, we define the main classes of singularities that we
will study and we recall the inclusions between these classes.  We
will proceed by introducing the smaller classes of singularities first
to then move to the wider ones.

\begin{definition}
Let $X$ be a $\qq$-Gorenstein normal algebraic variety 
and $\phi \colon Y \rightarrow X$ be a log resolution, 
i.e. a resolution such that the exceptional locus with
reduced scheme structure is purely divisorial and simple normal crossing.
Then we can write
\[ K_Y = \phi^*(K_X) + \sum_{i} a_iE_i,\]
where $E_i$ are pairwise different exceptional divisors over $X$. 
We say that $X$ is {\em log terminal}
(resp. {\em log canonical}) if $a_i>-1$ (resp. $a_i\geq -1$) for every
$i\in I$. 
\end{definition}

\begin{definition}
  Let $X$ be a normal algebraic variety and let
  $\phi \colon Y \rightarrow X$ be any resolution of singularities of
  $X$.  We say that $X$ has {\em rational singularities} if the higher
  direct image sheaves $R^{i}\phi_{*}\mathcal{O}_Y$ vanish for all
  $i>0$.
\end{definition}

Log terminal singularities are rational ~\cite[Theorem 5.22]{KM98} but
in general log canonical singularities are not rational.  Indeed, let
$X$ to be the affine cone over a planar elliptic curve $C$ in
$\mathbb{P}^2$. We can produce a log resolution
$\phi\colon Y \rightarrow X$ with only one exceptional divisor $E$ such
that $K_Y = \phi^*(K_X)-E$, so $X$ is log canonical but the stalk of
$R^1\phi_*\mathcal{O}_X$ at the vertex of the cone is isomorphic to
$H^1(C,\mathcal{O}_C)\simeq \kk$ which is non trivial and so $X$ is
not rational.

\begin{definition}
  We say that a commutative local Noetherian ring $R$ is {\em
    Cohen-Macaulay} if its depth is equal to its dimension.  An
  algebraic variety $X$ is {\em Cohen-Macaulay} if the local
  ring $\mathcal{O}_{X,x}$ is  Cohen-Macaulay for all $x\in X$.
\end{definition}

It is known that rational singularities in characteristic zero are
Cohen-Macaulay (see, e.g., ~\cite[Theorem 5.10]{KM98}), however there are many examples of log canonical
singularities which are not Cohen-Macaulay.

We will use the following characterization of Du Bois singularities
which are Cohen-Macaulay due to K\'{o}vacs-Schwede-Smith. \\

\begin{theorem}[{\cite[Theorem 1]{KSS}}]\label{caractdubois}
  Let $X$ be a normal Cohen-Macaulay algebraic variety, let
  $\phi \colon Y\rightarrow X$ be a log resolution and denote by $E$
  the exceptional divisor of $\phi$ with reduced scheme structure.
  Then $X$ has Du Bois singularities if and only if we have an
  isomorphism
  of sheaves $\phi_* \omega_Y(E)\simeq \omega_X$.
\end{theorem}

There is an analogous characterization of rational singularities in
terms of canonical sheaves: a Cohen-Macaulay normal algebraic variety
$X$ has rational singularities if and only if
the natural inclusion of sheaves $\phi_*\omega_Y \hookrightarrow \omega_X$ is an isomorphism, 
where $\phi \colon Y \rightarrow X$ is
a resolution of singularities \cite{KKMS73}. From
Theorem~\ref{caractdubois} and this characterization of rational
singularities it follows that rational singularities are Du Bois.
Indeed, by ~\cite[Lemma 3.14]{KSS} we have two natural injections
$\phi_*\omega_Y \hookrightarrow \phi_* \omega_Y(E) \hookrightarrow \omega_X$,
and Theorem~\ref{caractdubois} states that $X$ is Du Bois whenever the inclusion 
of sheaves $\phi_*\omega_Y(E)\hookrightarrow \omega_X$ is an isomorphism.
Moreover, it is known that log canonical singularities are Du
Bois~\cite{KK10}.

  It is known that toric singularities are log terminal
  ~\cite[Proposition 11.4.24]{CLS} and therefore they belong to all
  the above classes of singularities.  In higher complexity the
  situation becomes more complicated: rational singularities can be
  characterized in terms of divisorial fans ~\cite[Theorem 3.4]{LS},
  but there are not complete characterizations for the other classes
  of singularities defined above. Partial characterizations of
  Cohen-Macaulay and log terminal singularities with torus action are
  given in ~\cite{LS}.

\begin{remark}
The above definitions are independent of the chosen resolution
(or log resolution) of the normal variety $X$. 
By ~\cite{AW97} we may assume that the resolution of singularities is
$\mathbb{T}$-equivariant.
\end{remark}

\section{Preliminary results}\label{sec2}

In this section, we collect all the ingredients we need for the proof
of our main theorem. Some are borrowed from the literature on
$\TT$-varieties while some other are proven here. 

\subsection{$\mathbb{T}$-invariant divisors}\label{ss1}

We recall the description of the $\mathbb{T}$-invariant prime divisors
of the $\mathbb{T}$-varieties $\widetilde{X}(\mathcal{D})$ and
$X(\mathcal{D})$ following ~\cite[Proposition 3.13]{PS}.  Let
$\mathcal{D}$ be a p-divisor on $(Y,N)$, where $Y$ is a normal
algebraic variety, then any fiber of the good quotient $\pi$ over a
point of $Y$ which is not contained in $\supp(\mathcal{D})$ is
$\mathbb{T}$-equivariantly isomorphic to the toric variety
$X(\sigma(\mathcal{D}))$, therefore the $\mathbb{T}$-variety
$\widetilde{X}(\mathcal{D})$ admits an open subset which is isomorphic
to the product
$X(\sigma(\mathcal{D}))\times (Y - \supp(\mathcal{D}))$.

An {\em horizontal $\mathbb{T}$-divisor} is a $\mathbb{T}$-invariant
divisor of $\widetilde{X}(\mathcal{D})$ which dominates $Y$, these
divisors are in one to one correspondence with the rays $\rho$ of the
cone $\sigma(\mathcal{D})$ and they can be realized as the closure in
$\widetilde{X}(\mathcal{D})$ of the subvariety
$V(\rho)\times (Y-\supp(\mathcal{D}))$, where $V(\rho)$ is the toric
invariant divisor of $X(\sigma(\mathcal{D}))$ corresponding to the ray
$\rho$ of $\sigma(\mathcal D)$. On the other hand, we define a {\em vertical
  $\mathbb{T}$-divisor} to be a $\mathbb{T}$-invariant divisor of
$\widetilde{X}(\mathcal{D})$ which does not dominate $Y$, these
divisors are in one-to-one correspondence with pairs $(Z,v)$ where $Z$
is a prime divisor of $Y$ and $v$ is a vertex of the polyhedron
$\Delta_Z$.  We often write $\mathcal{V}(\Delta_Z)$ for the set of
vertices of the polyhedron $\Delta_Z$.  Any $\mathbb{T}$-invariant
divisor of $\widetilde{X}(\mathcal{D})$ is either horizontal or
vertical.

\begin{definition}
\label{defbig}
Now we turn to describe the horizontal and vertical
$\mathbb{T}$-divisors which are not contained in the exceptional locus
of the morphism
$r\colon \widetilde{X}(\mathcal{D})\rightarrow X(\mathcal{D})$.  
\begin{itemize}
\item
A ray $\rho$ of $\sigma(\mathcal{D})$ is {\em big} if
the $\qq$-divisor $\mathcal{D}(u)$ is big for every
$u\in \relint( \sigma^\vee \cap \rho^{\perp})$, the corresponding
horizontal $\mathbb{T}$-divisor will be called {\em big} as well.
\item
Given a prime divisor
 $Z\subset Y$ and a vertex $v\in \Delta_Z$ we say
that the vertex $v$ is a {\em big} if the $\qq$-divisor
$\mathcal{D}(u)|_Z$ is a big divisor for every $u$ in the interior of
the normal cone
\[
\mathcal{N}(\Delta_Z, v):=
\{
u \mid \langle u, w-v \rangle >0 
\text{ for every $w\in \Delta_Z$}
\}.
\]
As before, we say that the corresponding vertical $\mathbb{T}$-divisor
is big.  
\end{itemize}
Then we can say that the codimension one exceptional set of
$r$ is the union of all the $\mathbb{T}$-divisors which are not big.
\end{definition}

Finally we turn to introduce some notation for the horizontal and
vertical $\mathbb{T}$-divisors on $\widetilde{X}(\mathcal{D})$ and
$X(\mathcal{D})$.  Given a ray $\rho$ in $\sigma(\mathcal{D})$ we may
denote by $\widetilde{D}_\rho$ the corresponding horizontal divisor of
$\widetilde{X}(\mathcal{D})$ and whenever $\rho$ is big we denote by
$D_\rho$ its image $r(\widetilde{D}_\rho)$.  Given a prime divisor
$Z\subset Y$ and a vertex $v\in \mathcal{V}(\Delta_Z)$ we write
$\widetilde{D}_{(Z,v)}$ for the corresponding vertical divisor and
whenever $v$ is big we denote by $D_{(Z,v)}$ its image in
$X(\mathcal{D})$.

\subsection{$\mathbb{T}$-orbit decomposition}\label{ss2}

Now we turn to describe the $\mathbb{T}$-orbits of the
$\mathbb{T}$-varieties $\widetilde{X}(\mathcal{D})$ and
$X(\mathcal{D})$ following ~\cite[Section 10]{AH05}.  In order to do
so, we begin by defining the toric bouquet associated to a polytope
$\Delta\subset N_\qq$.

\begin{definition}
  Given a $\sigma$-polyhedron $\Delta \subset N_\qq$ we denote by
  $\mathcal{N}(\Delta)$ the {\em normal fan} of $\Delta$, which is the
  fan in $M_{\mathbb{Q}}$ with support $\sigma^\vee$ and whose cones
  corresponds to linearity regions of the function
  $\min\langle \Delta, - \rangle \colon \sigma^\vee \rightarrow \qq$.
  Observe that the cones of $\mathcal{N}(\Delta)$ are in one-to-one
  dimension-reversing correspondence with the faces of $\Delta$.
  Given a face $F$ of $\Delta$, we will denote by $\mathcal{N}(F)$ the
  corresponding cone of $\mathcal{N}(\Delta)$.  We define the {\em
    toric bouquet} of $\Delta$ to be
\[
X(\Delta):=\Spec( \kk [ \mathcal{N}(\Delta) \cap M]),
\]
where $\kk [ \mathcal{N}(\Delta) \cap M]$ equals
$\kk[\sigma^\vee \cap M]$ as $\kk$-vector spaces, but the
multiplication rule in $\kk[\mathcal{N}(\Delta)\cap M]$ is given by
\[
\chi^u \cdot \chi^{u'}:=
\left\{
	\begin{array}{cl}
		\chi^{u+u'}  & \text{if u,u' belong to a common cone of $\mathcal{N}$}, \\
		0 &  \text{otherwise}.
	\end{array}
\right. 
\]
\end{definition}

Given a point $y\in Y$, we define the {\em fiber polyhedron} to be
\[
\Delta_y :=
\left\{
	\begin{array}{cl}
		\sum_{y\in Z} \Delta_Z  & \text{if $y\in \supp(\mathcal{D})$}, \\
		\sigma(\mathcal{D}) &  \text{otherwise}.
	\end{array}
\right. 
\]
In the previous section we stated that the fiber $\pi^{-1}(y)$ over a
point $y\in Y$ which is not contained in the support of $\mathcal{D}$
is $\mathbb{T}$-equivariantly isomorphic to the toric variety
$X(\sigma(\mathcal{D}))$, in general the special fibers of $\pi$ are
not irreducible, but they can still be interpreted as toric bouquets
as follows.  Given a point $y\in \supp(\mathcal{D})$, the fiber
$\pi^{-1}(y)$ is $\mathbb{T}$-equivariantly isomorphic to the toric
bouquet $X(\Delta_y)$.  Therefore, we have a dimension-reversing
bijection between between the orbits of $\pi^{-1}(y)$ and the faces of
$\Delta_y$.  Thus, the orbits of $\widetilde{X}(\mathcal{D})$ are in
one-to-one dimension-reversing correspondence with the pairs $(y,F)$
where $y\in Y$ and $F$ is a face of $\Delta_y$.

Finally, we describe the $\mathbb{T}$-orbits
of $\widetilde{X}(\mathcal{D})$ which are identified by the contraction
morphism $r$.
Given an element $u\in \sigma(\mathcal{D})^\vee\cap M$ we have an
induced morphism
\[
\phi_u \colon Y \rightarrow Y_u, 
\text{ where }
Y_u=\Proj \left( \bigoplus_{n\in \zz_{\geq 0}} 
H^0(Y, \mathcal{O}_Y(\mathcal{D}(nu)))\right).
\]
From ~\cite[Theorem 10.1]{AH05}
we know that two orbits of $\widetilde{X}(\mathcal{D})$
corresponding to the pairs $(y,F)$ and $(y',F')$
are identified via the $\mathbb{T}$-equivariant contraction $r$
if and only if 
\[
\mathcal{N}(F)
=\mathcal{N}(F') \text{ and }
\phi_u(y)=\phi_u(y') \text{ for some $u\in \relint(\mathcal{N}(F))$}.
\]
Observe that in the case that $Y$ is $\pp^1$, then either $\phi_u$ is an isomorphism
or is the constant morphism, depending whether the divisor $\mathcal{D}(nu)$ has
positive or trivial degree, respectively.

\subsection{Sections of $\mathbb{T}$-invariant sheaves}\label{ss3}

In this subsection we describe the sections of the sheaf induced by a
$\mathbb{T}$-invariant Weil divisor on $X(\mathcal{D})$ (see, e.g., ~\cite[Section 3.3]{PS}).  Recall that
the prime $\mathbb{T}$-invariant divisors of $X(\mathcal{D})$ are of
the form $D_\rho$ or $D_{Z,v}$ for $\rho$ a big ray of
$\sigma(\mathcal{D})$ or $v\in \mathcal{V}(\Delta_Z)$ a big vertex.
Therefore, any $\mathbb{T}$-invariant Weil divisor $D$ of
$X(\mathcal{D})$ can be written as
\[
D=\sum_{\rho \text{ big }} a_\rho D_\rho
+ \sum_{\substack{(Z,v)\\ \text{ $v$ big }}} 
b_{Z,v}D_{Z,v}.
\]
Observe that the field of rational functions of
$\widetilde{X}(\mathcal{D})$ and the field of rational functions of
$X(\mathcal{D})$ are isomorphic to the fraction field of $\kk(Y)[M]$,
so any quasi-homogeneous rational function of
$\widetilde{X}(\mathcal{D})$ or $X(\mathcal{D})$ can be written as
$f\chi^u$ where $u\in M$ and $f\in \kk(Y)$.  By
~\cite[Proposition 3.14]{PS} we know that the principal divisor
associated to $f\chi^u$ on $\widetilde{X}(\mathcal{D})$ is given by
\begin{equation}\label{div-x}
\divv( f\chi^u)=
\sum_{\rho} \langle \rho, u\rangle D_\rho
+\sum_{(Z,v)} \mu(v)(\langle v, u \rangle +\ord_Z f) D_{Z,v},	
\end{equation} 
and the associated principal divisor on $X(\mathcal{D})$ is given by 
\begin{equation}\label{div-x-tilda}
\divv( f\chi^u)=
\sum_{\substack{\rho \\ \text{ $\rho$ big }}} \langle \rho, u\rangle D_\rho
+\sum_{\substack{(Z,v)\\ \text{ $v$ big }}} \mu(v)(\langle v, u \rangle +\ord_Z f) D_{Z,v},	
\end{equation} 
where $\mu(v)$ denotes the smaller positive integer such that $\mu(v)v$ is an integral point of $N_\qq$.
Therefore, the divisor 
$\divv(f\chi^u) + D$ is effective in 
$X(\mathcal{D})$ if and only if
\begin{eqnarray}
\langle \rho , u \rangle + a_\rho \geq 0, \text{ for every $\rho$ big, and }\nonumber\\
\langle v, u\rangle + \frac{b_{Z,v}}{\mu(v)} +\ord_Z f \geq 0, \text{ for every pair $(Z,v)$ with $v$ big}.\nonumber
\end{eqnarray}
We denote by $\square(D)$ the integer points of the polyhedron defined
by
\[ \{ u \in M_\qq \mid \langle \rho, u \rangle \geq -a_\rho,
\text{ for each $\rho$ big } \}.\]
Given an element $u\in \square(D)\cap M$
we denote by $\psi_D(u)$ the divisor of $Y$ given by
\[
\sum_{Z\subset Y}\left( \min_{\substack{ v\in \mathcal{V}(\Delta_Z)\\ \text{ $v$ big }}}
 \left( \langle v, u \rangle +  \frac{b_{Z,v}}{\mu(v)} \right) \right)Z .
\]
Therefore, we have that
\[
H^0 \left(X(\mathcal{D}),\mathcal{O}_{X(\mathcal{D})}(D)\right)
\simeq \bigoplus_{u\in \square(D)\cap M}
H^0 \left(X(\mathcal{D}),\mathcal{O}_{X(\mathcal{D})}(D)\right)_u,
\]
and 
\[
H^0\left(X(\mathcal{D}),\mathcal{O}_{X(\mathcal{D})}(D)\right)_u 
\simeq H^0(Y, \mathcal{O}_Y(\psi_D(u)))\chi^u.
\]

\subsection{Canonical sheaf}\label{ss4}

Let $K_Y$ be a canonical divisor of $Y$. Given a subvariety
$Z\subset Y$ we denote by $\coef_Z K_Y$ the coefficient of $K_Y$ along
$Z$. Then, we can describe the canonical divisor of
$\widetilde{X}(\mathcal{D})$ as
\begin{equation}\label{eq:canonical-tilde}
K_{\widetilde{X}(\mathcal{D})}= \sum_{(Z,v)} (\mu(v)\coef_Z K_Y +
\mu(v)-1)\widetilde{D}_{Z,v} -\sum_\rho \widetilde{D}_\rho\,,
\end{equation} 
and the canonical divisor of $X(\mathcal{D})$ as
\begin{equation}\label{eq:canonical-aff}
K_{X(\mathcal{D})} = \sum_{\substack{(Z,v)\\ \text{ $v$ big }}}
(\mu(v)\coef_Z K_Y + \mu(v)-1)D_{Z,v} -\sum_{\rho \text{ big }} D_\rho\,.
\end{equation} 

\section{Proof of the main theorems}\label{sec25}

In this section, we prove the main theorems of the article.
We start by stating the theorems for rational singularities
and Du Bois singularities of $\mathbb{T}$-varieties.
In order to state the theorems, 
we will use the following notation on p-divisors:

\begin{notation}\label{not}
Let $\mathcal{D}$ be a p-divisor on $(Y,N)$.
In what follows, when we write $Z\subset Y$, we mean 
a prime divisor on $Y$.
Let $\mathcal{V}_Z$ be the set of vertices of $\Delta_Z$.
Here, $\Delta_Z$ is the polyhedral divisor of $\Delta$ at the prime divisor $Z\subset Y$.
Unless otherwise stated, $\rho$
is the primitive generator of an extremal ray of the cone $\sigma(\mathcal{D})$.

Let $\mathcal{D}$ be a p-divisor on $(Y,N)$.
Then, we have a function 
$\mathcal{D}: \sigma^\vee \rightarrow  \CaDiv_\qq(Y).
$
This function can be extended to the whole group $M$ as follows:
\[
\mathcal{D}(u) := \sum_{Z\subset Y} \min_{v\in \mathcal{V}_Z} \langle u, v\rangle Z.
\]
Note that the previous sum is finite as $\Delta_Z =\rho(\Delta)$ for all but finitely many prime divisors $Z$.
We define $Y^b \subset Y$ to be complement of the union of the prime divisors $Z\subset Y$ for which no $v\in \mathcal{V}_Z$ is big.
Now, we turn to introduce some variations of $\mathcal{D}$.
We define the function
\[
\mathcal{D}_\omega \colon M\rightarrow 
\CaDiv_\qq(Y_v)
\qquad
u\mapsto
\sum_{Z\subset Y} \min_{\underset{\text{$v$ big}}{v\in \mathcal{V}_Z}} \langle v, u\rangle Z.
\] 
Note that the previous 
divisor has support in $Y^b$.
On the other hand, we define:
\[
 \widetilde{\langle v,u\rangle} :=
 \begin{cases}
  \lceil \langle v,u\rangle \rceil & \text{ if $v$ is big}\\
  \lfloor \langle v,u \rangle\rfloor + 1 & \text{ if $v$ is not big}.
  \end{cases}
\]
Then, we can define the function 
\[
\mathcal{D}_{\widetilde{\omega},E} \colon M\rightarrow 
\CaDiv_\qq(Y) 
 \qquad
u\mapsto\sum_{Z\subset Y} \min_{v\in \mathcal{V}_Z} \widetilde{\langle v,u\rangle} Z.
\] 

\end{notation}

\begin{theorem}\label{thm:rat}
Let $\mathcal{D}$ be a p-divisor on $(Y,N)$ with smooth support.
Assume that $X(\mathcal{D})$ is Cohen-Macaulay.
Then $X(\mathcal{D})$ has rational singularities if and only if 
the following conditions hold:
\begin{enumerate} 
\item For every $u\in M$ with 
    $\langle \rho, u \rangle \geq 1$ 
    for all $\rho$, 
    the isomorphism 
    \[
    H^0(Y^b,\omega_{Y^b}(\lceil \mathcal{D}_\omega(u)\rceil)) \simeq 
    H^0(Y,\omega_Y(\lceil \mathcal{D}(u)\rceil)) 
    \] 
    holds, and 
    \item 
    for every $u\in M$ with 
    $\langle \rho, u \rangle \geq 1$ for all $\rho$ big and
    $\langle \rho,u\rangle \leq 0$ for at least one $\rho$ not big, 
    the equality
    \[
    h^0(Y^b,\omega_{Y^b}(\lceil \mathcal{D}_\omega(u)\rceil)) =0
    \] 
    holds.
\end{enumerate}
\end{theorem}

\begin{theorem}\label{thm:DB}
Let $\mathcal{D}$ be a p-divisor on $(Y,N)$ with smooth support.
Assume that $X(\mathcal{D})$ is Cohen-Macaulay.
Then $X(\mathcal{D})$ has Du Bois singularities if and only if 
the following conditions hold:
\begin{enumerate}
    \item For every $u\in M$ with 
    $\langle \rho, u \rangle \geq 1$ for all $\rho$ big and
    $\langle \rho,u\rangle \geq 0$ for all $\rho$ not big, 
    the isomorphism 
    \[
    H^0(Y^b,\omega_{Y^b}(\lceil\mathcal{D}_\omega(u)\rceil)) \simeq 
    H^0(Y,\omega_Y(\mathcal{D}_{\widetilde{\omega},E}(u))) 
    \] 
    holds, and 
    \item 
    for every $u\in M$ with 
    $\langle \rho, u \rangle \geq 1$ for all $\rho$ big and
    $\langle \rho,u\rangle \leq -1$ for at least one $\rho$ not big, 
    the equality
    \[
    h^0(Y^b,\omega_{Y^b}(\lceil \mathcal{D}_\omega(u)\rceil)) =0
    \] 
    holds.
\end{enumerate}
\end{theorem}

\begin{lemma}\label{lem:cont}
Let $\mathcal{D}$ be a
simple normal crossing proper polyhedral divisor on a projective variety $Y$.
We denote by $r\colon \widetilde{X}(\mathcal{D})\rightarrow X(\mathcal{D})$ the associated $\mathbb{T}$-equivariant birational contraction.
Let $E$ be the reduced exceptional divisor of $r$.
Assume $X(\mathcal{D})$ is Cohen-Macaulay.
Then, the following statements hold:
\begin{enumerate}
    \item $X(\mathcal{D})$ has rational singularities
    if and only if
    $r_*\omega_{\widetilde{X}(\mathcal{D})} \simeq \omega_{X(\mathcal{D})}$, and
    \item $X(\mathcal{D})$ has Du Bois singularities if and only if
    $r_*\omega_{\widetilde{X}(\mathcal{D})}(E)\simeq \omega_{X(\mathcal{D})}$.
\end{enumerate}
\end{lemma}

\begin{proof}
The proof of the first statement is analogous to the proof of the second statement. Hence, we only prove the second statement.
For simplicity, we denote 
$X(\mathcal{D})$ by $X$
and $\widetilde{X}(\mathcal{D})$ by $\widetilde{X}$.

Since the polyhedral divisor $\mathcal{D}$ has simple normal crossing support, then the variety
$\widetilde{X}$ has toroidal singularities (see, e.g.,~\cite[Example 2.5]{LS}).
Let $E^+$ be the reduced sum of all the torus invariant vertical divisors mapping to $\supp(\mathcal{D})$ plus the reduced sum of all the torus invariant horizontal divisors.
Again by~\cite[Example 2.5]{LS}, the pair $(\widetilde{X},E^+)$ is toroidal.
In particular, the pair $(\widetilde{X},E^+)$ has log canonical singularities (see, e.g.~\cite[Corollary 11.4.25]{CLS}).
Let $E$ be the reduced exceptional divisor of $\widetilde{X}\rightarrow X$.
Note that $E\leq E^+$.
Let $\phi\colon X'\rightarrow X$ be a log resolution of $\widetilde{X}$
and $\tilde{\phi}\colon X'\rightarrow \tilde{X}$ be the induced projective birational morphism.
Let $E'$ be the reduced exceptional divisor of $X'\rightarrow X$.
By~\cite[Theorem 1]{KSS}, we have that $X$ is Du Bois if and only if 
\[
\phi_*\omega_{X'}(E') \simeq \omega_X.
\] 
Hence, it suffices to prove that
\begin{equation}\label{omega-push} 
\tilde{\phi}_*\omega_{X'}(E') \simeq 
\omega_{\tilde{X}}(E).
\end{equation} 
The isomorphism~\eqref{omega-push} follows from~\cite[Lemma 3.15]{KSS}.
\end{proof}

\begin{proof}[Proof of Theorem~\ref{thm:DB}]
By Lemma~\ref{lem:cont}, 
it suffices to check that 
$r_*\omega_{\widetilde{X}(\mathcal{D})}(E) \simeq \omega_{X(\mathcal{D})}$.
Observe that both sheaves
$r_*\omega_{\widetilde{X}(\mathcal{D})}(E)$
and $\omega_{X(\mathcal{D})}$
are $\mathbb{T}$-invariant 
sheaves on an affine $\mathbb{T}$-variety.
Hence, it suffices to prove that the $M$-graded pieces
of 
\begin{equation}\label{eq:1st} 
H^0(\widetilde{X}(\mathcal{D}),\omega_{\widetilde{X}(\mathcal{D})}(E))
\end{equation} 
agree with the $M$-graded pieces of
\begin{equation}\label{eq:2nd}
H^0(X(\mathcal{D}),\omega_{X(\mathcal{D})}).
\end{equation}

First, we compute the $M$-graded pieces of the group~\eqref{eq:1st}.
Let $f\chi^u \in \kk(Y)[M]$.
By~\eqref{div-x-tilda} and~\eqref{eq:canonical-tilde} , we know that $f\chi^u$ belongs to~\eqref{eq:1st} if and only if the following conditions hold:
\begin{align}
\label{c1}
  \langle \rho, u \rangle \geq 1 & \text{ for all $\rho$ big }  \\
\label{c2}
  \langle \rho, u \rangle \geq 0 & 
  \text{ for all $\rho$ not big} \\
\label{c3}
  {\rm ord}_Z(f) + {\rm coeff}_Z(K_Y)  + \frac{\mu(v)-1+
  \varepsilon(v)}{\mu(v)} + \langle v,u\rangle \geq 0
  &\text{ for all $(Z,v)$},
\end{align}
where $\varepsilon(v) = 0$ if $v$ is big and
$\varepsilon(v) = 1$ otherwise.
Hence, $f\chi^u$ belongs to~\eqref{eq:1st} if and only if
\eqref{c1}, \eqref{c2} are satisfied and~\eqref{c3}
is replaced by 
\[
    {\rm ord}_Z(f)  + {\rm coeff}_Z(K_Y) + 
    \min_{v\in \mathcal{V}_Z}
    \left\{ 
    \frac{\mu(v)+1-\varepsilon(v)}{\mu(v)} + \langle v,u \rangle
    \right\} 
    \geq 0 \quad \text{ for all $Z$}. 
\]
On the other hand
\[
 \left\lfloor \frac{\mu(v)+1-\varepsilon(v)}{\mu(v)} + \langle v,u \rangle\right\rfloor
 =
 \begin{cases}
  \lceil  \langle v,u\rangle \rceil & \text{ if $v$ is big}\\
  \lfloor \langle v,u \rangle \rfloor +1 & \text{ if $v$ is not big}\\
 \end{cases} 
\]
we conclude that 

\[
H^0(\widetilde{X}(\mathcal{D}),\omega_{\widetilde{X}(\mathcal{D})}(E)) \simeq 
\bigoplus_{\underset{\langle \rho, u\rangle \geq 0 \text{ $\forall$ $\rho$ not big}}{\langle \rho,u\rangle \geq 1 \text{ $\forall$ $\rho$ big}}} H^0(Y,\omega_Y(
\mathcal{D}_{\widetilde{\omega},E}(u))).
\]

Here, the function $\mathcal{D}_{\widetilde{\omega},E}$ is defined as in Notation~\ref{not}.
Now, we turn to compute the $M$-graded
pieces of the group~\eqref{eq:2nd}.
By~\eqref{div-x} and~\eqref{eq:canonical-aff},
we know that $f\chi^u$ belongs to 
$H^0(X(\mathcal{D}),\omega_{X(\mathcal{D})})$
if and only if~\eqref{c1} holds and~\eqref{c3}
holds for any $v$ big. The latter condition is 
equivalent to
\[
  {\rm ord}_Z(f)  + 
  {\rm coeff}_Z(K_Y) + 
  \min_{\underset{\text{$v$ big}}{v\in \mathcal{V}_Z}}
  \left\{ 
  \frac{\mu(v)-1}{\mu(v)} + \langle u,v\rangle 
  \right\} \geq 0 \quad \text{ for all $Z$}.
\]
Hence, we conclude that 
\begin{equation*}
    H^0(X(\mathcal{D}),\omega_{X(\mathcal{D})}) \simeq \bigoplus_{\langle \rho, u\rangle \geq 1 \text{ $\forall$ $\rho$ not big}}
    H^0(Y^b,\omega_{Y^b}( \lceil \mathcal{D}_\omega(u))\rceil ).
\end{equation*}
Thus, the affine variety $X(\mathcal{D})$ has Du Bois singularities if and only if the following two conditions are satisfied: 
\begin{itemize}
    \item For every $u\in M$ with 
    $\langle \rho, u \rangle \geq 1$ for all $\rho$ big and
    $\langle \rho,u\rangle \geq 0$ for all $\rho$ not big, 
    the isomorphism 
    \[
    H^0(Y^b,\omega_{Y^b}(\lceil \mathcal{D}_\omega(u)\rceil )) \simeq 
    H^0(Y,\omega_Y(\mathcal{D}_{\widetilde{\omega},E}(u))) 
    \] 
    holds.
    \item 
    For every $u\in M$ with 
    $\langle \rho, u \rangle \geq 1$ for all $\rho$ big and
    $\langle \rho,u\rangle \leq -1$ for at least one $\rho$ not big, 
    the equality
    \[
    h^0(Y^b,\omega_{Y^b}(\lceil \mathcal{D}_\omega(u)\rceil )) =0
    \] 
    holds.
\end{itemize}
This finishes the proof of the theorem.
\end{proof}

\begin{proof}[Proof of Theorem~\ref{thm:rat}]
By Lemma~\ref{lem:cont}, 
it suffices to check that 
$r_*\omega_{\widetilde{X}(\mathcal{D})} \simeq \omega_{X(\mathcal{D})}$.
Observe that both sheaves
$r_*\omega_{\widetilde{X}(\mathcal{D})}$
and $\omega_{X(\mathcal{D})}$
are $\mathbb{T}$-invariant 
sheaves on an affine $\mathbb{T}$-variety.
Hence, it suffices to prove that the $M$-graded pieces
of 
\begin{equation}\label{eq:3rd} 
H^0(\widetilde{X}(\mathcal{D}),\omega_{\widetilde{X}(\mathcal{D})})
\end{equation} 
agree with the $M$-graded pieces of
\[
H^0(X(\mathcal{D}),\omega_{X(\mathcal{D})}).
\]
In view of the proof of Theorem~\ref{thm:DB}, it suffices to find the $M$-graded pieces of~\eqref{eq:3rd}.
By~\eqref{div-x-tilda} and~\eqref{eq:canonical-tilde} , we know that $f\chi^u$ belongs to
$H^0(\widetilde{X}(\mathcal{D}),\omega_{\widetilde{X}(\mathcal{D})}(E))$
if and only if the following conditions hold:
\begin{align} 
\label{c4}
  \langle \rho, u \rangle \geq 1 & \text{ for all $\rho$}  \\
\label{c5}
  {\rm coeff}_Z(K_Y)  + \frac{\mu(v)-1}{\mu(v)} + \langle v,u\rangle + {\rm ord}_Z(f) \geq 0 & \text{ for all $(Z,v)$}.
\end{align}
Hence, $f\chi^u$ belongs to~\eqref{eq:3rd} if and only if
\eqref{c4} holds and~\eqref{c5} is replaced by
\[
  {\rm ord}_Z(f)  + 
  {\rm coeff}_Z(K_Y) + 
  \min_{v\in \mathcal{V}_Z}
  \left\{ 
  \frac{\mu(v)-1}{\mu(v)} + \langle u,v\rangle 
  \right\} \geq 0 \quad \text{ for all $Z$}.
\]
We conclude that 
\begin{equation}\label{eq: }
    H^0(\widetilde{X}(\mathcal{D}),\omega_{\widetilde{X}(\mathcal{D})}) \simeq \bigoplus_{\langle \rho, u\rangle \geq 1 \text{ $\forall \rho$}}
    H^0(Y,\omega_Y(\lceil\mathcal{D}(u)\rceil)).
\end{equation}
Thus, the affine variety $X(\mathcal{D})$ has rational singularities if and only if the following two conditions are satisfied: 
\begin{itemize}
    \item For every $u\in M$ with 
    $\langle \rho, u \rangle \geq 1$ for all $\rho$, 
    the isomorphism 
    \[
    H^0(Y^b,\omega_{Y^b}(\lceil \mathcal{D}_\omega(u)\rceil )) \simeq 
    H^0(Y,\omega_Y(\lceil \mathcal{D}(u)\rceil)) 
    \] 
    holds.
    \item 
    For every $u\in M$ with 
    $\langle \rho, u \rangle \geq 1$ for all $\rho$ big and
    $\langle \rho,u\rangle \leq 0$ for at least one $\rho$ not big, 
    the equality
    \[
    h^0(Y^b,\omega_{Y^b}(\lceil \mathcal{D}_\omega(u)\rceil)) =0
    \] 
    holds.
\end{itemize}
This finishes the proof of the theorem.
\end{proof}

Now, we are ready to prove the theorems in the case of complexity one.

\begin{proof}[Proof of Theorem~\ref{thm:comp-1-rat}]
If $\mathcal{D}$ is a p-divisor on $(\pp^1,N)$, then all the vertices $v\in \mathcal{V}_{p_i}$ are big.
We conclude that
$\pp^{1,b} = \pp^1$ and 
$\mathcal{D}_\omega =\mathcal{D}$.
Hence, the condition Theorem~\ref{thm:rat}.(1) is vacuous.
On the other hand, the condition
of Theorem~\ref{thm:rat}.(2) is equivalent to:
for every $u\in M$ with $\langle \rho, u\rangle \geq 1$ for all $\rho$ big
and $\langle \rho,u\rangle \leq 0$ for at least one $\rho$ not big, we have $\deg \lceil \mathcal{D}(u) \rceil \leq 1$.
This is equivalent to the statement of the theorem after replacing $u$ with $-u$.
\end{proof}

\begin{proof}[Proof of Theorem~\ref{thm:comp-1-DB}]
If $\mathcal{D}$ is a p-divisor on $(\pp^1,N)$, then all the vertices $v\in \mathcal{V}_{p_i}$ are big.
We conclude that
$\pp^{1,b} = \pp^1$ and 
$\mathcal{D}_{\widetilde{\omega},E} = \lceil \mathcal{D}\rceil$.
Hence, the condition Theorem~\ref{thm:DB}.(1) is vacuous.
On the other hand, the condition
of Theorem~\ref{thm:DB}.(2) is equivalent to:
for every $u\in M$ with $\langle \rho, u\rangle \geq 1$ for all $\rho$ big
and $\langle \rho,u\rangle \leq -1$ for at least one $\rho$ not big, we have $\deg \lceil \mathcal{D}(u) \rceil \leq 1$.
This is equivalent to the statement of the theorem after replacing $u$ with $-u$.
\end{proof}

\begin{example} \label{example}
  Let $M=\ZZ^2$ so that $N=\ZZ^2$ and $M_\QQ=N_\QQ=\QQ^2$. We also let
  $\sigma$ be the first quadrant, i.e., the cone spanned by $(1,0)$
  and $(0,1)$. Taking $Y=\mathbb{P}^1$, we let $\DD$ be the p-divisor
  on $Y$ given by
  $\DD=\Delta_1\cdot z_1+\Delta_2\cdot z_2+\Delta_3\cdot
  z_3+\Delta_4\cdot z_4$, where
  $\Delta_1=\Delta_2=\left(\nicefrac{1}{2},1\right)$ and
  $\Delta_3=\Delta_4=\left(-\nicefrac{1}{2},1\right)$. Since
  $\deg \DD=\sum_{i=1}^4 \Delta_i=(0,4)+\sigma\subsetneq \sigma$ 
  it follows that $\DD$ is a p-divisor.

  There are two rays in $\sigma$. We denote the ray spanned by $(1,0)$
  by $\rho_1$ and the ray spanned by $(0,1)$ by $\rho_2$. We have that
  $\rho_1$ is big since it does not intersect $\deg\DD$ while $\rho_2$
  is not big since it intersects $\deg\DD$.

  By \cite[Theorem~5]{Pet} we have that  $X(\DD)$ is a $\ZZ/2\ZZ$ quotient of $X(\DD')$, where
  $\DD'=\Delta'_1\cdot z_1+\Delta'_2\cdot z_2+\Delta'_3\cdot
  z_3+\Delta'_4\cdot z_4$, where
  $\Delta_1=\Delta_2=\left(1,1\right)$ and
  $\Delta_3=\Delta_4=\left(-1,1\right)$ 
  which is a toric variety since $\DD'$ is equivalent to a p-divisor with at most two non-trivial coefficients. This yields $X(\DD')$ is Cohen-Maculay and so is $X(\DD)$  by \cite[Proposition~13]{HE}.  Furthermore, the set
  \begin{align}
    \{ u\in M \mid 
    \langle u, \rho \rangle \leq -1 \text{ for all $\rho$ big and } 
    \langle u,\rho \rangle \geq 0 
    \text{ for at least one $\rho$ not big}
    \}
  \end{align}
  in Theorem~\ref{thm:comp-1-rat} is $\{(u_1,u_2)\in M\mid u_1\leq -1
  \mbox{ and } u_2\geq 0\}$. For instance, $u=(-1,0)$
  in contained in this set and $\deg\lfloor\DD(u)\rfloor=-2$. Theorem~\ref{thm:comp-1-rat} implies that $X(\DD)$ does not have rational singularities. This agrees with \cite[Proposition~5.1]{LS}

  On the other hand, the set
  \begin{align}
    \{ u\in M \mid 
    \langle u, \rho \rangle \leq -1 \text{ for all $\rho$ big and } 
    \langle u,\rho \rangle \geq 1 
    \text{ at least one $\rho$ not big}
    \}
  \end{align}
  in Theorem~\ref{thm:comp-1-DB} is $\{(u_1,u_2)\in M\mid u_1\leq -1
  \mbox{ and } u_2\geq 1\}$. Now, for every $u=(u_1,u_2)$ in this set,
  we have
  $$\deg\lfloor\DD(u)\rfloor=2\left(\left\lfloor\frac{1}{2}u_1\right\rfloor+\left\lfloor-\frac{1}{2}u_1\right\rfloor\right)+4u_2\geq
  -2+4u_2\geq 2\geq -1\,.$$
  We conclude from Theorem~\ref{thm:comp-1-DB} that $X(\DD)$ has Du Bois singularities.
\end{example}

\end{document}